\documentclass[12pt,a4paper]{amsart}
\usepackage{latexsym,amsfonts,amsthm,amsmath,mathrsfs,amssymb}
\usepackage[british]{babel}
\usepackage[dvips]{color}
\usepackage[utf8]{inputenc}
\usepackage[T1]{fontenc}
\usepackage[dvips,final]{graphics}
\usepackage{xcolor}
\usepackage{mathrsfs}
\usepackage{mathtools} 
\usepackage{breqn}
\usepackage{epstopdf}
\usepackage{verbatim}
\usepackage{amscd}
\usepackage{hyperref}
\usepackage[english,capitalize]{cleveref}
\usepackage{aliascnt}

\usepackage{todonotes}

\makeatletter
\def\newaliasedtheorem#1[#2]#3{
	\newaliascnt{#1@alt}{#2}
	\newtheorem{#1}[#1@alt]{#3}
	\expandafter\newcommand\csname #1@altname\endcsname{#3}
}
\makeatother

\theoremstyle{plain}

\newtheorem{theorem}{Theorem}[section]
\newaliasedtheorem{prop}[theorem]{Proposition}
\newaliasedtheorem{lem}[theorem]{Lemma}
\newaliasedtheorem{coroll}[theorem]{Corollary}

\theoremstyle{definition}
\newaliasedtheorem{defi}[theorem]{Definition}
\newaliasedtheorem{quest}[theorem]{Question}
\newaliasedtheorem{fact}[theorem]{Fact}

\theoremstyle{remark}
\newaliasedtheorem{rem}[theorem]{Remark}
\newaliasedtheorem{exa}[theorem]{Example}

\newcommand{\R}{\mathbb{R}}
\newcommand{\N}{\mathbb{N}}

\newcommand{\LL}{\mathbb{L}}

\newcommand{\C}{\mathbb{C}}

\newcommand{\mm}{\mathfrak {m}}

\let\altphi\phi
\let\phi\varphi
\let\varphi\altphi
\let\altphi\undefined

\newcommand{\average}{{\mathchoice {\kern1ex\vcenter{\hrule height.4pt
width 6pt
depth0pt} \kern-9.7pt} {\kern1ex\vcenter{\hrule height.4pt width 4.3pt
depth0pt}
\kern-7pt} {} {} }}

\address{\textsc{Daniela Di Donato}: 
Dipartimento di Ingegneria Industriale e Scienze Matematiche, Via Brecce Bianche, 12 60131 Ancona, Universit\'a Politecnica delle Marche.}
\email{d.didonato@staff.univpm.it}

\title{Intrinsic Cheeger energy for the intrinsically Lipschitz constants}

\date{\today}

\author{ Daniela Di Donato}

\subjclass[]{ 
26A16  
51F30 
54E35 
}
\keywords{Cheeger energy, Lipschitz graphs, vector space, Metric spaces}

\begin{document}

\begin{abstract} Recently, in metric spaces, Le Donne and the author introduced the so-called intrinsically Lipschitz sections. The main aim of this note is to adapt Cheeger theory for the classical Lipschitz constants in our new  context. More precisely,	we define the intrinsic Cheeger energy from $L^2(Y,\R^s)$  to $[0,+\infty],$ where $(Y,d_Y,\mm)$ is a metric measure space  and we characterize it in terms of a suitable notion of relaxed slope. In order to get this result, in more general context, we establish some properties of the intrinsically Lipschitz constants like the Leibniz formula, the product formula and the upper semicontinuity of the asymptotic intrinsically Lipschitz constant. 
\end{abstract}

\maketitle 
\tableofcontents

\section{Introduction}
Starting from the seminal papers by Franchi, Serapioni and Serra Cassano \cite{FSSC, FSSC03, MR2032504}, Le Donne and the author \cite{DDLD21} introduced and studied the so-called intrinsically Lipschitz sections in the metric spaces which generalize the intrinsically Lipschitz maps in the FSSC sense. The purpose of this note is to introduce and give some properties of  the Cheeger energy \cite{C99} in our intrinsic context. The main result is Theorem \ref{theorem30may}.  The long-term objective is to obtain \textbf{a Rademacher type theorem \'a la Cheeger } \cite{C99} for the intrinsically Lipschitz sections.
 
  This project fits to an active line of research \cite{Pansu, Vittone20, FMS14} on Rademacher type theorems, but it also studies connections to different notions and mathematical areas like Cheeger energy, Sobolev spaces, Optimal transport theory. The reader can see \cite{AGS08.2, AGS08.3, CEMCS, S06, MR2480619, MR2459454, H96, S00, K04, DM, HKST15}. 
 
In metric measure spaces, there are different approaches to define the Cheeger energy; one of those is represented by the theory of relaxed slopes. Here, we change a bit the classical notion in order to define the intrinsic Cheeger energy (see Section \ref{Cheeger energy of a metric measure space}).

Here we use a similar technique exploited in \cite[Lemma 2.2.4]{P17} and  \cite[Lemma 4.3]{ACDM15}.  A basic difference with the classical case is given by the fact we use a $non$ $usual$ Leibniz formula for the intrinsic Lipschitz constants. Conversely the usual case, we cannot obtain a result without additional condition on the fibers of $\pi$ and we get a non linear formula (see Theorem \ref{sommadiLipeLip.2senzaa}). On the other hand, we didn't expect the classic formula to work because we proved the Lipschitz constant for $\phi$ and for $\lambda \phi$ with $\lambda \ne 0$ is the same (see Proposition \ref{propconvex30}). Yet, it is no possible to get the usual Sobolev space using the intrinsic slope because we don't have homogeneity property. Moreover, it is interesting to underline that our corresponding triangle inequality can be true with a small constant between 1/2 and 1 in particular cases (cf. \eqref{equationLeibnitz} for $c\in [1,2)$).

 Thanks to the study of the Leibniz formula, we find a suitable setting in order that
\begin{equation*}
\phi, \psi \mbox{ are intrinsic Lipschitz} \quad \Longrightarrow \quad \phi + \psi \mbox{ is intrinsic Lipschitz}.
\end{equation*}
To prove this result (see Proposition \ref{sommadiLipeLip.1}), we use basic mathematical tools. Moreover, it is relevant because it is not known whether the sum of two Lipschitz maps in the FSSC sense is so too in subRiemannian Carnot groups \cite{ABB, BLU, CDPT} which are particular metric spaces. However, our result does not include Carnot groups since the projection map is not linear in this specific setting. On the other hand, it is no too restrictive to ask the linearity of a projection map.

 The rest of the paper is organized as follows. In $\mathbf{Section \, 2},$ we recall the definition of the intrinsically Lipschitz sections and we present some properties which we will use later. In $\mathbf{Section \, 3},$ we provide some basic properties of the intrinsic Lipschitz constants like the upper semicontinuity of the asymptotic intrinsically Lipschitz constant (see Proposition \ref{propSEMICONT}) and the product formula (see Proposition \ref{propslope.20may}).   
$\mathbf{Section \, 4}$ states the sum of two intrinsic Lipschitz sections is so too and the Leibniz formula is proved.  $\mathbf{Section \, 5}$  is the heart of this paper. Here, we define the intrinsic Cheeger energy and we give a characterization of it in terms of an appropriate notion of intrinsically relaxed slope.

\section{Intrinsically Lipschitz sections} In \cite{DDLD21}, we give the following notion. 
\begin{defi}\label{def_ILS} 
Let $(X,d)$ be a metric space, $Y$ be a topological space and $\pi:X\to Y$ be a quotient map, i.e.,  it is continuous, open, and surjective. We say that a map $\phi :Y \to X$ is a section of $\pi $ if 
\begin{equation}\label{equation1}
\pi \circ \phi =\mbox{id}_Y.
\end{equation}
Moreover, we say that a map $\phi:Y\to X$ is an {\em intrinsically Lipschitz section of $\pi$ with constant $L$},  with $L\in[1,\infty)$, if in addition
\begin{equation}\label{equationFINITA}
d(\phi (y_1), \phi (y_2)) \leq L d(\phi (y_1), \pi ^{-1} (y_2)), \quad \mbox{for all } y_1, y_2 \in Y.
\end{equation}
Here $d$ denotes the distance on $X$, and, as usual, for a subset $A\subset X$ and a point $x\in X$, we have
$d(x,A):=\inf\{d(x,a):a\in A\}$.
\end{defi}
A first observation is that we study a sort of biLipschitz condition; indeed, since $\phi (y_2) \in \pi ^{-1} (y_2)$ it holds
\begin{equation*}
d(\phi (y_1), \pi ^{-1} (y_2)) \leq d(\phi (y_1), \phi (y_2)) \leq L d(\phi (y_1), \pi ^{-1} (y_2)), \quad \mbox{for all } y_1, y_2 \in Y.
\end{equation*}
 Indeed, we underline that, in the case  $ \pi$ is a Lipschitz quotient or submetry \cite{MR1736929, Berestovski},  being intrinsically Lipschitz  is equivalent to biLipschitz embedding, see Proposition 2.4 in \cite{DDLD21}. Moreover, since $\phi$ is injective by \eqref{equation1}, the class of Lipschitz sections not include the constant maps.   

%
%

Finally, we present a result which will use later.

 \begin{prop}[Proposition 3.1-3.2 \cite{D22.1}]\label{propconvex30}  
Let $\pi: X\to Y$ be a linear and quotient map between a normed space $X$ and a topological space $Y.$ 
\begin{enumerate}
\item the set of all sections is a vector space over $\R$ or $\C.$
\item If $\phi: Y \to X$ is an intrinsically Lipschitz section of $\pi,$ then for any $\lambda \in \R-\{0\}$ the section $\lambda \phi$ is also intrinsic Lipschitz for $1/\lambda \pi$ with the same Lipschitz constant.

\end{enumerate}

\end{prop}

 {\bf Acknowledgements.}  We would like to thank Giuseppina Autuori for helpful suggestions regarding the spaces $L^p(Y, \R^s, \mm)$.

\section{Intrinsic   Lipschitz  constants}
 We adapt the theory of  \cite{C99, DM} (see also \cite{K04}) in our intrinsic case. The following definitions are studied in \cite{D22.29}. Here, we give new results like the product formula and the upper semicontinuity of the asymptotic in\-trin\-si\-cally Lipschitz constant. 

\begin{defi}\label{def_ILS.1} Let $\phi:Y\to X$ be a section of $\pi$. Then we define 
\begin{equation*}
ILS (\phi):= \sup _{\substack{y_1, y_2 \in Y \\ y_1\ne y_2}} \frac{d(\phi (y_1), \phi (y_2))}{  d(\phi (y_1), \pi ^{-1} (y_2)) } \in [1,\infty ]
\end{equation*}
and
\begin{equation*}
\begin{aligned}
ILS (Y,X,\pi ) &:= \{ \phi :Y \to X \,:\, \phi \mbox{ is an intrinsically Lipschitz section of $\pi$ and }  ILS(\phi) < \infty \},\\
ILS_{b} (Y,X,\pi) & := \{ \phi \in  ILS (Y,X,\pi) \,:\, \mbox{spt}(\phi) \mbox{ is bounded} \}.\\
ILS_{0} (Y,X,\pi) & := \{ \phi \in  ILS (Y,X,\pi) \,:\, \mbox{spt}(\phi) \mbox{ is compact} \}.\\
\end{aligned}
\end{equation*}
For simplicity, we will write $ILS (Y,X)$ instead of  $ILS (Y,X,\pi ).$
\end{defi}

\begin{defi}\label{def_ILS.2} Let $\phi:Y\to X$ be a  section of $\pi$. Then we define the local intrinsically Lipschitz constant (also called intrinsic slope) of $\phi$ the map $Ils (\phi):Y \to [1,+\infty )$ defined as 
\begin{equation*}
Ils (\phi) (z):= \limsup _{y\to z} \frac{d(\phi (y), \phi (z))}{  d(\phi (y), \pi ^{-1} (z)) },
\end{equation*}
if $z \in Y$ is an accumulation point; and $Ils (\phi) (z):=0$ otherwise.
\end{defi}

\begin{defi}\label{def_ILS.3} Let $\phi:Y\to X$ be a section of $\pi$. Then we define the asymptotic  intrinsically Lipschitz constant of $\phi$ the map $Ils_a (\phi):Y \to [1,+\infty )$ given by
\begin{equation*}
Ils_a (\phi) (z):= \limsup _{y_1,y_2\to z}\frac{d(\phi (y_1),\phi (y_2))}{  d(\phi (y_1), \pi ^{-1} (y_2)) }
\end{equation*}
if $z \in Y$ is an accumulation point and $Ils (\phi) (z):=0$ otherwise.
\end{defi}

 \begin{rem}\label{defrem} Notice that by $\phi (y_2) \in \pi ^{-1} (y_2),$ it is trivial that $ d(\phi (y_1), \pi ^{-1} (y_2)) \leq d(\phi (y_1), \phi (y_2))$ and so $Ils(\phi) \geq 1.$ Moreover, it holds 
\begin{equation*}
 Ils (\phi ) \leq Ils_a (\phi ) \leq ILS(\phi).
\end{equation*}
\end{rem}

Following Proposition 2.1.11 in  \cite{P17}, we prove the upper semicontinuity of the asymptotic intrinsically Lipschitz constant.
\begin{prop}\label{propSEMICONT}
Let $\phi :Y \to X$ be a section of a linear and quotient map $\pi:X \to Y$ between two metric spaces. Then, the map $y \mapsto Ils_a(\phi)(y)$ is upper semicontinuous with respect to the metric topology of $Y,$ i.e.,
\begin{equation*}
\limsup _{ z\to y} Ils_a(\phi) (z) \leq Ils_a (\phi) (y), \quad \forall y\in Y.
\end{equation*}

\end{prop}

\begin{proof} Fix $y\in Y.$ 
We want to prove that if $(y_n)_n$ is a sequence in $Y$ such that $d_Y(y_n , y) \to 0,$ then
\begin{equation*}
\limsup _{ n\to \infty} Ils_a(\phi) (y_n) \leq Ils_a (\phi) (y).
\end{equation*}
By definition, it holds
\begin{equation*}
Ils_a (\phi) (y_n) = \inf _{\varepsilon >0} \sup _{\substack{z_1, z_2 \in B(y_n, \varepsilon ) \\ z_1\ne z_2}}    \frac{d(\phi (z_1),\phi (z_2))}{  d(\phi (z_1), \pi ^{-1} (z_2)) }.
\end{equation*}
Hence, fix $\varepsilon >0$ and we show that $$B\left (y_n, \frac \varepsilon 3\right) \subset B(y, \varepsilon),\quad \mbox{ for $n$ big enough.} $$  Indeed, since $y_n$ converge to $y$ we have that there is $N \in \N$ such that for any $n\geq N$ it holds $ d(y_n , y ) <\frac \varepsilon 3;$ on the other hand, if $z\in B(y_n, \frac \varepsilon 3)$ with $n\geq N$ we deduce that
\begin{equation*}
d(z,y) \leq d(z, y_n) + d(y_n, y) < \varepsilon,
\end{equation*}
i.e., $z \in B(y, \varepsilon).$ As a consequence, for such $n\in \N$ it follows
\begin{equation*}
Ils_a (\phi) (y_n) \leq  \sup _{\substack{z_1, z_2 \in B(y_n, \frac \varepsilon 3) \\ z_1\ne z_2}}    \frac{d(\phi (z_1),\phi (z_2))}{  d(\phi (z_1), \pi ^{-1} (z_2)) } \leq  \sup _{\substack{z_1, z_2 \in B(y, \varepsilon ) \\ z_1\ne z_2}}    \frac{d(\phi (z_1),\phi (z_2))}{  d(\phi (z_1), \pi ^{-1} (z_2)) }.
\end{equation*}
Eventually taking the limsup in $n$ and then the infimum in $\varepsilon$, we get the thesis.
\end{proof}

It is no possible to get the same statement for the intrinsically slope but we can define its semicontinuous envelope as follows for $X=\R$.
 \begin{defi}
Let $f:Y \to \R.$ We denote by $f^*$ ($f_*$ respectively) the upper semicontinuous (lower semicontinuous) envelope of $f, $ namely:
\begin{equation*}
\begin{aligned}
f^* (y) := \inf \{ g (y) \,:\, g \geq f , \, g (.) \mbox{ upper semicontinuous}\},\\
f _* (y) := \inf \{ g (y) \,:\, g \geq f, \, g (.) \mbox{ lower semicontinuous}\}.\\
\end{aligned}
\end{equation*}
\end{defi}

\begin{rem}\label{remperchegeerdopo}
 Obviously if $\phi(.)$ is already upper (lower) semicontinuous, its envelope coincides with $\phi$ itself. Moreover, 
 \begin{equation*}
 ILS(\phi) \geq   Ils_a (\phi )(y) \geq Ils (\phi)^* (y).
\end{equation*}
The first inequality comes from the definition of $ ILS(\phi),$ while the second one is due to the fact that the asymptotic Lipschitz constant of a function is upper semicontinuous (see Proposition \ref{propSEMICONT}) and bigger or equal than the slope.
 \end{rem}
    
  \subsection{Product for the intrinsic constants} Here, we give a condition in order to get a formula for the product of asymptotic intrinsically Lipschitz constants. It is easy to see that the same statement holds for the intrinsic slope.
         \begin{prop}[Product of asymptotic intrinsically Lipschitz constants]\label{propslope.20may} Let $Y$ be  a metric space and $\pi :\R^s \to Y$ be a linear and quotient map. Assume also that  $\phi $ and $ \psi $ are intrinsically $L$-Lipschitz sections of $\pi$ bounded by $M$ such that
         \begin{equation}\label{equation20maggio}
 \frac{\min  \{d(\phi  (z), \pi ^{-1} (y)), d(\psi  (z), \pi ^{-1} (y)) \}}{d(\phi(z)\psi (z) , \pi ^{-1} (y))} \leq k, \quad   \forall z,y \in Y.
\end{equation}
Then
\begin{equation}\label{equationLeibnitz.prod.20MAY}
Ils_a (\phi \psi )( y) \leq Mk Ils_a (\phi) ( y)+ Mk Ils_a (\psi) ( y), \quad \forall y\in Y.
\end{equation}

\end{prop}

 \begin{proof} 
Fix $\varepsilon >0$ and $y\in Y.$ Notice that  for every $z_1, z_2 \in B(y, \varepsilon )$ with $z_1\ne z_2$ we have
\begin{equation*}
\begin{aligned}
d(\phi (z_1) \psi (z_1), \phi( z_2) \psi ( z_2))& \leq d(\phi (z_1) \psi (z_1), \phi( z_1) \psi ( z_2)) + d(\phi (z_1) \psi ( z_2), \phi( z_2) \psi (z_2))\\
& \leq M d( \psi (z_1), \psi (z_2)) +M d(\phi (z_1), \phi(z_2))\\
\end{aligned}
\end{equation*}
Hence, dividing for  $d(\phi (z_1) \psi (z_1), \pi ^{-1} (z_2))$    we obtain
 \begin{equation*}
\begin{aligned}
& \frac{ d(\phi (z_1) \psi (z_1), \phi( z_2) \psi ( z_2)) }{d(\phi (z_1) \psi (z_1), \pi ^{-1} (z_2)) } \\
& \leq  M \frac{   d(\phi (z_2), \phi (z_1))}{d(\phi  (z_1), \pi ^{-1} (z_2))}  \frac{  d(\phi  (z_1), \pi ^{-1} (z_2)) }{d(\phi (z_1) \psi (z_1), \pi ^{-1} (z_2))} +
M \frac{ d(\psi (z_2),\psi (z_1))}{d(\psi (z_1), \pi ^{-1} (z_2))}   \frac{  d(\psi (z_1), \pi ^{-1} (z_2)) }{d(\phi (z_1) \psi (z_1), \pi ^{-1} (z_2)) }  \\
 & \leq  M k\frac{   d(\phi (z_2), \phi (z_1))}{d(\phi  (z_1), \pi ^{-1} (z_2))} +
Mk \frac{ d(\psi (z_2),\psi (z_1))}{d(\psi (z_1), \pi ^{-1} (z_2))} , \\
\end{aligned}
\end{equation*}
where in the second equality we used  the hypothesis \eqref{equation20maggio}.

Now taking the supremum in $z_1, z_2 \in B( y , \varepsilon)$ and then the infimum in $\varepsilon$,  we get the thesis \eqref{equationLeibnitz.prod.20MAY}.
\end{proof}

%

\section{Sum of intrinsically Lipschitz sections} 
\subsection{Sum of intrinsically Lipschitz sections is so too}\label{Sum of intrinsically Lipschitz sections is so too}  In this section, we prove that the sum of Lipschitz section is so too. This is relevant essentially because, in the case of the Carnot groups, but actually already in the model case of the Heisenberg groups, it is not known whether the sum of two Lipschitz maps in the sense of FSSC is still Lipschitz. However we ask that $\pi$ is linear and in Carnot groups this fact is not true. 

We notice that in order to obtain our result, we need a suitable parallel property of the fibers (see Lemma \ref{lem19}) which is trivial when $\pi$ is a submetry (and so intrinsic Lipschitz section is equivalent to ask biLipschitz embedding).

 \begin{theorem}\label{sommadiLipeLip.1} 
 Let  $Y$ be a metric space and $\pi: \R^s \to Y$ be a linear and quotient map. Then, the sum of two intrinsically Lipschitz sections is so too. 
\end{theorem}

\begin{proof}
 The thesis follows using basic mathematical tools. Indeed, for any $y_1, y_2 \in Y$ with $y_1\ne y_2$ the inequality 
\begin{equation}\label{sommadiLipeLip.1dis}
d(\eta(y_1), \eta (y_2)) \leq L d(\eta (y_1) , 1/2 \pi^{-1} (y_2) ),
\end{equation}
can be described in the following way: suppose that the trivial case $d(\eta(y_1), \eta (y_2)) \ne d(\eta (y_1) , 1/2 \pi^{-1} (y_2) )$ is not true; then, we have a right triangle where 
\begin{enumerate}
\item $\eta(y_1), \eta (y_2)$ are two vertices;
\item $d(\eta(y_1), \eta (y_2))$  is the hypotenuse; 
\item  either $d(\eta (y_1) , 1/2 \pi^{-1} (y_2) )$ is a cathetus or it is a larger segment of the cathetus with vertex $\eta (y_1)$;
\item  $L= 1/ \sin \alpha$ where $\alpha \in (0,\pi /2)$ is the angle between the two fix sides: the hypotenuse and the other cathetus (i.e., the segment which connects $\eta (y_2)$ and the point of minimal distant belong to the fiber).
\end{enumerate}
 The important point now is that $\alpha \ne 0$ and so there is a sufficiently large and finite constant $L$ which satisfies \eqref{sommadiLipeLip.1dis}. 
\end{proof}

Now we want to obtain the Leibniz formula for the  intrinsic slope using an additional suitable condition on the fibers.

 \begin{theorem}[Leibniz formula]\label{sommadiLipeLip.2senzaa} 
 Let $Y$ be a metric space and $\pi: \R^s \to Y$ be a linear and quotient map. Assume also that $\phi $ and $ \psi $ are intrinsically $L$-Lipschitz sections of $\pi$ such that there is $c \geq 1$ satisfying  
 \begin{equation}\label{equationSobolev10}
d(\pi ^{-1} (y),\pi ^{-1} (z) ) \geq \frac 1 c d(f (y), \pi ^{-1} (z)), \quad \, \forall y,z \in Y,
\end{equation}
$\mbox{for } f=\phi , \psi .$  Then, denoting $\eta = \alpha \phi + \beta \psi $ the map $Y\to \R^s$ with $\alpha , \beta \in \R-  \{ 0 \},$  we have that
\begin{equation}\label{equationLeibnitz.7}
Ils (\eta )( y) \leq   c /2  (  Ils (\phi) ( y)+   Ils (\psi) ( y)),\quad \forall y\in Y.
\end{equation}

\end{theorem}

We need the following lemma which is true in more general setting.
 \begin{lem}\label{lem19}
Let $X$ be a normed space, $Y$ be a topological space and $\pi: X \to Y$ be a linear and quotient map. Then
\begin{equation}\label{equation6magI}
| \lambda | d(\pi ^{-1} (y_1),\pi ^{-1} (y_2) ) =  d((1/\lambda  \pi )^{-1} (y_1), (1/\lambda  \pi )^{-1} (y_2) ), \quad \forall y_1, y_2 \in Y, \forall \lambda \in \R -\{0\}.
\end{equation}
\end{lem}

 \begin{proof}
 Fix $y_1, y_2 \in Y$ such that $y_1 \ne y_2.$ If $d(\pi ^{-1} (y_1),\pi ^{-1} (y_2) ) =d(a,b)$ for some $a\in \pi ^{-1} (y_1)$ and $b\in \pi ^{-1} (y_2),$ then a similar computation as in \cite[Theorem 2.10]{D22.2} we deduce that
 \begin{equation*}
\lambda  a\in (  1 / \lambda  \pi )^{-1} (y_1) \quad \mbox{and}\quad \lambda  b\in   (1/ \lambda  \pi )^{-1} (y_2),
\end{equation*}
and, consequently,
 \begin{equation}\label{equation6mag1}
d(\pi ^{-1} (y_1),\pi ^{-1} (y_2) ) = \|a-b\| = \frac 1 {|\lambda |}   \|\lambda a-\lambda b\| \geq \frac 1 {|\lambda |}  d((1/\lambda  \pi )^{-1} (y_1),(1/\lambda  \pi )^{-1} (y_2) ). 
\end{equation} 
 On the other hand, if $d((1/\lambda \pi )^{-1} (y_1), (1/\lambda  \pi )^{-1} (y_2) ) =d(s,t)$ for some $s\in (1/\lambda  \pi) ^{-1} (y_1)$ and $t\in (1/\lambda  \pi )^{-1} (y_2),$ then
 \begin{equation*}
\frac 1 \lambda  s\in  \pi ^{-1} (y_1) \quad \mbox{and}\quad \frac 1 \lambda  t\in  \pi ^{-1} (y_2),
\end{equation*}
and so
\begin{equation}\label{equation6mag2}
d((1/\lambda  \pi )^{-1} (y_1), (1/\lambda  \pi )^{-1} (y_2) ) = \|s-t\|=|\lambda | \left\| \frac 1 \lambda  s - \frac 1 \lambda  t \right\| \geq |\lambda |  d(\pi ^{-1} (y_1),\pi ^{-1} (y_2) ). 
\end{equation}
Hence, putting together \eqref{equation6mag1} and \eqref{equation6mag2}, we get \eqref{equation6magI}.
 \end{proof}

 \begin{proof}[Proof of Theorem $\ref{sommadiLipeLip.2senzaa}$]
Notice that since $Ils (\phi)=Ils (\lambda \phi)$ for any $\lambda \ne 0$ (see Proposition \ref{propconvex30} (2)), it is sufficient to prove \eqref{equationLeibnitz.7} for $\eta = \phi + \psi.$ 

We split the proof in two steps. In the first step we prove that $\eta$ is intrinsic Lipschitz and in the second one we show \eqref{equationLeibnitz.7} following similarly to \cite[Lemma 3.2]{KM16}. 

(1). The thesis follows from Theorem \ref{sommadiLipeLip.1}. 

(2). Fix $\bar y\in Y$ and $r>0.$ Suppose we are given $\varepsilon >0.$ For any $y\in B(\bar y ,r)$ we have
\begin{equation}\label{equationSobolev10.1}
\begin{aligned}
\frac{d(\phi (\bar y), \phi (y))}{  d(\phi (\bar y), \pi ^{-1} (y)) } \leq Ils (\phi)(\bar y) + \varepsilon \qquad{and} \mbox \qquad \frac{d(\psi (\bar y), \psi (y))}{  d(\psi (\bar y), \pi ^{-1} (y)) } \leq Ils (\psi)(\bar y) + \varepsilon .
\end{aligned}
\end{equation}
On the other hand, thanks to  Theorem \ref{sommadiLipeLip.1}, we can find $z\in B(\bar y, r)$ so that
\begin{equation*}
\begin{aligned}
Ils (\phi +\psi )(\bar y) \leq \frac{d(\phi (\bar y)+ \psi (\bar y), \phi (z)+\psi (z))}{  d(\phi (\bar y)+\psi (\bar y), (1/2\pi )^{-1} (z)) } + \varepsilon ,
\end{aligned}
\end{equation*}
and so 
\begin{equation*}
\begin{aligned}
 Ils (\phi +\psi )(\bar y) & \leq \frac{d(\phi (\bar y), \phi (z)) + d( \psi (\bar y), \psi (z))}{  d(\phi (\bar y)+\psi (\bar y), (1/2\pi )^{-1} (z)) } +\varepsilon \\
  & \leq \frac{d(\phi (\bar y), \phi (z)) + d( \psi (\bar y), \psi (z))}{  d((1/2\pi )^{-1} (\bar y), (1/2\pi )^{-1} (z)) } +\varepsilon \\
  & = \frac{d(\phi (\bar y), \phi (z)) + d( \psi (\bar y), \psi (z))}{2 d( \pi ^{-1} (\bar y), \pi ^{-1} (z)) }+\varepsilon \\
  & \leq \frac c 2  \frac{d(\phi (\bar y), \phi (z)) }{ d( \phi (\bar y), \pi ^{-1} (z)) } +  \frac c 2  \frac{ d( \psi (\bar y), \psi (z))}{ d( \psi (\bar y), \pi ^{-1} (z)) }+\varepsilon \\
  & \leq  \frac c 2 ( Ils (\phi)(\bar y) + \varepsilon ) +  \frac c 2  ( Ils (\psi)(\bar y) + \varepsilon )+\varepsilon
\end{aligned}
\end{equation*}
where in the first equality we used the triangle inequality, and in the second one we used the simple fact $d((1/2\pi )^{-1} (\bar y) , (1/2\pi )^{-1} (z))\leq  d(\phi (\bar y)+\psi (\bar y) , (1/2\pi )^{-1} (z)),$ noting that $\phi (\bar y)+\psi (\bar y) \in (1/2\pi )^{-1} (\bar y)$. In the first equality we used Lemma \ref{lem19} for $\lambda =2$ and in the last two inequalities we used \eqref{equationSobolev10} and \eqref{equationSobolev10.1}.

Hence, by the arbitrariness of $\varepsilon,$ the proof is complete.
\end{proof}

  \begin{prop}\label{sommadiLipeLip.2}
Under the same assumption of Theorem \ref{sommadiLipeLip.2senzaa}, denoting $\eta =  \alpha \phi +\beta  \psi $ with $\alpha , \beta \in \R -\{0\}$ the map $Y\to \R^s$  we have that
\begin{equation}\label{equationLeibnitz}
Ils_a (\eta )( y) \leq   c /2  (  Ils _a(\phi) ( y)+   Ils _a(\psi) ( y)),\quad \forall y\in Y.
\end{equation}
\end{prop}

\begin{proof}
The proof is equal to Theorem \ref{sommadiLipeLip.2senzaa}, noting that for any $z_1,z_2\in B(\bar y ,r)$ we have
\begin{equation*}
\begin{aligned}
\frac{d(\phi (z_1), \phi (z_2))}{  d(\phi (z_1), \pi ^{-1} (z_2)) } \leq Ils _a(\phi)(\bar y) + \varepsilon 
\end{aligned}
\end{equation*} and we can find $z_1, z_2\in B(\bar y, r)$ so that
\begin{equation*}
\begin{aligned}
Ils_a (\phi +\psi )(\bar y) \leq \frac{d(\phi (z_1)+ \psi (z_1), \phi (z_2)+\psi (z_2))}{  d(\phi (z_1)+\psi (z_1), (1/2\pi )^{-1} (z_2)) } + \varepsilon .
\end{aligned}
\end{equation*}

\end{proof}

  \begin{rem}
   Let $X$ be a normed space, $Y$ be a topological space and $\pi: X \to Y$ be a linear and quotient map. If $\phi , \psi :Y \to X$ are intrinsic Lipschitz sections satisfying \eqref{equationSobolev10}, then 
      \begin{equation*}
\begin{aligned}
& d( \alpha \phi(y) + \beta \psi (y), (1/\lambda \pi) ^{-1} (z))  \\
& \qquad \qquad   \leq 2c\,   \max\{|\alpha |d((1/\alpha   \pi )^{-1} (y), (1/\alpha  \pi )^{-1} (z) ) , |\beta |d((1/\beta   \pi )^{-1} (y), (1/\beta \pi )^{-1} (z) )\},\\
\end{aligned}
\end{equation*}
for every $y,z \in Y$ with $\alpha , \beta \in \R- \{ 0 \}$ such that $\alpha + \beta =\lambda.$  
   In particular, when $\alpha =\beta = 1/2,$ we have
   \begin{equation*}
\begin{aligned}
d( \phi(y) +\psi (y), (1/2\pi) ^{-1} (z)) &   \leq 2c\,   d((1/2  \pi )^{-1} (y), (1/2 \pi )^{-1} (z) ),\\
\end{aligned}
\end{equation*}
for every $y,z \in Y.$

Indeed, fix $y,z \in Y$ such that $y\ne z$ and let $a \in (1/\alpha \pi) ^{-1} (z)$ and $b \in (1/\beta \pi )^{-1} (z)$ such that $d(\alpha \phi(y), (1/\alpha \pi )^{-1} (z)) = d( \alpha \phi(y), a)$ and $d(\beta  \psi(y), (1/\beta \pi )^{-1} (z)) =d( \beta \psi(y), b).$ It is easy to see that $a+b \in (1/\lambda \pi) ^{-1} (z)$ and so
 \begin{equation*}
\begin{aligned}
& d( \alpha \phi(y) + \beta \psi (y), (1/\lambda \pi) ^{-1} (z))\\ & \quad  \leq \|\alpha \phi(y) -a\| + \|\beta \psi (y)-b \| \\
& \quad \leq  2 \max\{ d(\alpha  \phi(y), (1/\alpha \pi )^{-1} (z)) ,d( \beta \psi(y),(1/\beta  \pi )^{-1} (z)) \}\\
&  \quad \leq 2 c\, \max\{|\alpha |d((1/\alpha   \pi )^{-1} (y), (1/\alpha  \pi )^{-1} (z) ) , |\beta |d((1/\beta   \pi )^{-1} (y), (1/\beta \pi )^{-1} (z) )\},\\
\end{aligned}
\end{equation*}
Notice that in the last inequality we used the fact
  \begin{equation*}
\begin{aligned}
 d((1/\alpha \pi )^{-1}(y), (1/\alpha \pi )^{-1} (z)) & =|\alpha| d (\pi ^{-1} (y) , \pi ^{-1} (z)) \geq \frac {|\alpha |}  c d(\phi (y), \pi ^{-1} (z))\\ &  = \frac 1 {c}  d(\alpha  \phi(y), (1/\alpha \pi )^{-1} (z)),
\end{aligned}
\end{equation*}
where in the first equality we used Lemma \ref{lem19} and in the first inequality we used \eqref{equationSobolev10}.
\end{rem}

We conclude this section noting that the set of all sections satisfying the condition \eqref{equationSobolev10} is convex.

   \begin{prop}\label{lem22b} Let $Y$ be a metric space and $\pi: \R^s \to Y$ be a linear and quotient map. Then the set of all intrinsically $L$-Lipschitz sections $\phi $ and $ \psi $ of $\pi$  satisfying  \eqref{equationSobolev10} for some $c\geq 1$ is a convex set.
     \end{prop}
     
     We need the following lemma.
   \begin{lem}\label{lem22a}
   Let $X$ be a normed space, $Y$ be a topological space and $\pi: X \to Y$ be a linear and quotient map. Then
\begin{equation}\label{equ22may.1}
| \lambda | d(\phi (y_1),\pi ^{-1} (y_2) ) =  d(\lambda  \phi  (y_1), (1/\lambda  \pi )^{-1} (y_2) ), \quad \forall y_1, y_2 \in Y, \forall \lambda \in \R -\{0\}.
\end{equation}
   \end{lem}

 \begin{proof}
 Fix $y_1, y_2 \in Y,$ and $ \lambda \in \R -\{0\}.$ We consider $a\in \pi ^{-1} (y_2)$ such that $ d(\phi (y_1),\pi ^{-1} (y_2) ) = d(\phi (y_1),a ).$ By $\lambda a \in  (1/\lambda  \pi )^{-1}(y_2)$ we deduce that
 \begin{equation*}
\begin{aligned}
 d(\phi (y_1),\pi ^{-1} (y_2) ) & =\|\phi (y_1)-a\|=\frac 1 {|\lambda |} \|\lambda \phi (y_1)-\lambda a\| \geq \frac 1 {|\lambda |}  d(\lambda  \phi  (y_1), (1/\lambda  \pi )^{-1} (y_2) ).
\end{aligned}
\end{equation*}
Moreover, if we consider $s \in (1/\lambda  \pi )^{-1} (y_2) $ such that $ d(\lambda  \phi  (y_1), (1/\lambda  \pi )^{-1} (y_2) )= d(\lambda  \phi  (y_1), s)$ and recall that $1/\lambda s \in  \pi ^{-1} (y_2),$ then
  \begin{equation*}
 d(\lambda  \phi  (y_1), (1/\lambda  \pi )^{-1} (y_2) ) = \|\lambda  \phi  (y_1) - s \|= |\lambda |  \|  \phi  (y_1) - 1/\lambda s \| \geq  |\lambda |  d(\phi (y_1),\pi ^{-1} (y_2) ).
\end{equation*}
Putting together the last two inequalities, we get \eqref{equ22may.1}.
   \end{proof}

\begin{proof}[Proof of Proposition \ref{lem22b}] Fix $t\in [0,1]$ and $\phi , \psi.$ We want to prove that the section $t \phi +(1-t) \psi $ satisfies \eqref{equationSobolev10}, i.e., 
  \begin{equation*}
c\, d(\pi^{-1} (y),\pi ^{-1} (z) ) \geq  d \left(t \phi (y)+(1-t) \psi   (y), \pi ^{-1} (z) \right), \quad \, \forall y,z \in Y.
\end{equation*}
Fix $y,z \in Y$ such that $y\ne z$ and let $a  \in (1/t\pi )^{-1} (z)$ and $ b  \in (1/(1-t)\pi )^{-1} (z)$ such that $d \left(t \phi (y), (1/t\pi )^{-1} (z) \right) =d \left(t \phi (y), a \right) $ and $d \left((1-t) \psi (y), (1/(1-t)\pi )^{-1} (z) \right) = d \left((1-t)\psi (y), b \right).$ Notice that $a+b \in \pi ^{-1} (z),$ we have that
 \begin{equation*}
\begin{aligned}
 d \left(t \phi (y)+ (1-t) \psi   (y), \pi ^{-1} (z) \right) &  \leq  \left\| t \phi (y) -  a \right\|+ \left\| (1-t) \psi   (y)- b \right\|\\
 & = d \left(t \phi (y), (1/t\pi )^{-1} (z) \right) +  d \left((1-t) \psi   (y),(1/(1-t) \pi )^{-1} (z) \right) \\
 & = t  d(\phi (y),\pi ^{-1} (z) ) +(1-t)  d(\psi (y),\pi ^{-1} (z) ) \\
 & \leq c d(\pi ^{-1} (y),\pi ^{-1} (z) ),
\end{aligned}
\end{equation*}
where in the second equality we used Lemma \ref{lem22a} and in the last inequality we used \eqref{equationSobolev10}.
   \end{proof}

\subsection{Applications to groups}
In this section, we will apply the theory developed in Section \ref{Sum of intrinsically Lipschitz sections is so too} to the case of groups.  
The general setting is a group $G$ together with a closed subgroup $H$ of $G$ in such a way that the quotient space  $G/H:=\{gH:g\in G\}$ naturally is a topological space for which the  map $\pi:g\mapsto gH$ is a quotient map.

A section for the map $\pi: G\to G/H$ is just a map $\phi: G/H \to G$ such that $\phi(gH)\in gH$, since we point out the trivial identity $\pi^{-1}(gH)= \pi^{-1}(nH)=nH$.  To have the notion of intrinsic Lipschitz section we need the group $G$ to be equipped with a distance which we assume left-invariant. We refer to such a $G$ as a {\em metric group}. When $H$ is $1$-dimensional we get the following result.

 \begin{coroll}\label{sommadiLipeLip.Carnot} 
 Let  $G= N \rtimes H$ be a metric group with $H$ one-dimensional and $\pi :G \to G/H$ linear and Lipschitz map. Then, the sum of two intrinsically Lipschitz maps in the FSSC sense is so too. 
\end{coroll}

  \begin{proof} It is enough to consider Proposition 1.6 in \cite{DDLD21} and Theorem \ref{sommadiLipeLip.1}. 
\end{proof}

   \section{Intrinsic Cheeger energy of a metric measure space}\label{Cheeger energy of a metric measure space} 
  In this section we use the same notation of the last one, except for the constant $c$ which appearing in \eqref{equationSobolev10}. Here $$ c\geq 2.$$ This assumption guarantees us that the constant appearing in Proposition \ref{sommadiLipeLip.2} and so in Leibniz formula is greater  or equal to  one.  Here we use $L^q(Y, \mm)$ and $L^q(Y, \R^s, \mm)$ in order to consider the maps $Y \to \R$ and $Y \to \R^s,$ respectively. 

\subsection{The spaces $L^q(Y, \R^s, \mm)$} It is no usual to consider the space  $L^q(Y, \R^s, \mm);$ however, here it is fundamental because, in our intrinsic context, if $X=\R$ then $Y\subset \R$ and so we get the trivial case since $\pi$ is also injective. Hence, we must study the case $X=\R^s$ with $s>1.$ Because of this, we define the space $L^q(Y, \R^s, \mm)$ for $q\in [1, + \infty).$ In a natural way,  we ask that every component of a map $\psi =(\psi_1, \dots, \psi _s): Y \to \R^s$ belongs to  $L^q(Y,  \mm)$ (see,  for instance, \cite{B10}). More precisely,
\begin{defi}
Let  $\pi:\R^s \to Y$ be  a quotient map. We define the set
\begin{equation*}
\begin{aligned}
L^q(Y, \R^s, \mm) & = \{ \psi =(\psi_1, \dots, \psi _s) :Y\to \R^s \,:\,  \psi \mbox{ is a section of } \pi \\
&\quad \quad   \mbox{ and } \psi_i \in  L^q(Y,  \mm) \mbox{ for any } i=1,\dots , s\},
\end{aligned}
\end{equation*}
and we endow it with the following norm
\begin{equation*}
\|\psi\|_q =\sum_{i=1}^s |\psi_i|_q,
\end{equation*}
where $|.|_q$ is the usual norm in $L^q(Y, \mm).$ 
\end{defi}

We give a short proof regarding the fact that $\|.\|_q$  is a norm.
\begin{enumerate}
\item $\|\psi\|_q\geq 0$ by definition of  $|.|_q;$
\item $\|\psi\|_q= 0$ if and only if $|\psi _i|_q =0$ for every $i=1,\dots , s$ and so iff $\psi \equiv 0;$
\item By $| \lambda \psi _i|_q = | \lambda  | | \psi _i|_q$ for any $\lambda \in \R,$ we have that $\|\lambda \psi\|_q  =|\lambda |\|\psi\|_q $  for any $\lambda \in \R;$
\item  For any $\psi , \eta \in L^q(Y, \R^s, \mm),$ we get
\begin{equation*}
\|\psi +\eta \|_q \leq  | \psi _1|_q +|\eta _1|_q +\dots +  | \psi _s|_q+| \eta _s|_q = \|\psi  \|_q + \| \eta \|_q,
\end{equation*}
i.e.,  the triangle inequality holds.
\end{enumerate}

Notice that there is possible to define other norms in  $L^q(Y, \R^s, \mm)$ as 
\begin{enumerate}
\item $\|\psi \|'_q: =\max _{i=1,\dots , s} |\psi _i|_q;$
\item $\|\psi \|''_q:= \sqrt{\sum_{i=1}^s |\psi_i|_q^2}.$
\end{enumerate}
 However, in this paper, we use just the convergence property of this class of maps and so we give the following definition.
\begin{defi}
Let  $(\psi_{n})_{n\in \N} :=(\psi_{1,n}, \dots , \psi_{s,n})_{n\in \N} \subset L^q(Y, \R^s, \mm)$ and $\psi :=(\psi _1,\dots,  \psi _s) \in L^q(Y, \R^s,  \mm).$ We define the convergence of the sequence $(\psi_{n})_{n\in \N}$ on $L^q(Y, \R^s, \mm)$ as the convergence of every component $\psi_{i,n}$, i.e.,
\begin{equation*}
\psi_{i,n} \to \psi_i, \quad \mbox{ in } L^q(Y, \mm),
\end{equation*}
for any $i=1,\dots , s.$ We denote it as
\begin{equation*}
\psi _n \to \psi  \,\, \mbox{ in } L^q(Y, \R^s, \mm). 
\end{equation*}
\end{defi}
Actually,  the convergence  is up to a no zero constant but this is not important because  it appears when we use a $non$ $usual$ Leibniz formula  and  the intrinsically Lipschitz constants of $\psi$ and of $\alpha \psi$ (with $\alpha \ne 0$) are the same and so not depend on $\alpha$ (see Proposition \ref{propconvex30}   (2)).


\subsection{Intrinsic relaxed slope}

  \begin{defi}[Intrinsic relaxed slope]\label{defIntrinsically relaxed slope} Let $q\in (1,+\infty)$ and $\phi \in L^q(Y,\R^s, \mm)$  be a section of   $\pi :\R^s \to Y .$ We say that a non negative map $G \in L^q(Y, \mm)$ is an intrinsically relaxed slope of $\phi$ if there are $H_1, H_2 \in L^q(Y, \mm)$ such that
  \begin{enumerate}
\item $0\leq H_1(y) \leq  G(y)$ for a.e. $y\in Y.$ 
\item  There is a sequence $ (\phi _h)_h \subset \LL $ such that
\begin{equation*}
\phi _h \to \phi  \,\, \mbox{ in } L^q(Y, \R^s, \mm) \quad \mbox{ and } \quad  Ils_a ( \phi_h) (.) \rightharpoonup  H_2 \,\, \mbox{ in } L^q(Y, \mm),
\end{equation*}
\item $0\leq H_1(y) \leq H_2(y)$ for a.e. $y\in Y.$
\end{enumerate}
 
\end{defi}

Note that, by definition, $H_2$ is also an  intrinsic relaxed slope of $\phi $ with the same $H_1$ of $G$ and itself. 

The main difference with respect to the usual definition of  relaxed slope is that we have $H_1\ne H_2.$ This is because our Leibniz formula is non linear.

  \begin{defi}[Set of intrinsic $q$-relaxed slopes]\label{defIntriRquadro} Let $q\in (1,+\infty)$ and $\phi \in L^q(Y, \mm)$  be a section of   $\pi .$ We denote by $R^q(\phi)$ the set:
\begin{equation*}
R^q(\phi) := \{ G \in L^q(Y, \mm) \, :\, G(\cdot ) \mbox{ is an intrinsic relaxed slope of $\phi$}\}.
\end{equation*}
\end{defi}

An important point is that there is always a minimal element inside the class $R^q(\phi)$ even in a pointwise sense, when it is no empty. 
       
  \begin{prop}[Locality and Minimality]\label{Locality and Minimality} Let $q\in (1,+\infty)$ and $\phi \in L^q(Y,\R^s, \mm)$  be a section of   $\pi .$  Then, the following properties hold:
  \begin{enumerate}
\item $R^q(\phi)$ is closed and convex subset of $L^q(Y, \mm).$
\item If $G_1, G_2 \in R^q(\phi)$, then $\min\{ G_1, G_2\} \in R^q(\phi).$
\end{enumerate}
As a consequence, the $L^q(Y, \mm)$-minimal element in the class $R^q(\phi)$ is well defined and we denote by $|D\phi|_q.$ Moreover,
\begin{equation*}
|D\phi|_q \leq G , \quad \mm-a.e. \mbox{ in } Y, \quad \forall G \in R^q(\phi ).
\end{equation*}
\end{prop}

    \begin{proof} 
    Proposition \ref{sommadiLipeLip.2} and \ref{propconvex30}   easily imply that $R^q(\phi)$ is a closed and convex family of sections. Indeed, let $\{G_h\} _h \subset R^q(\phi)$ such that $G_h \to G$  in $L^q(Y, \mm).$ In order to prove the closed property, we want to show that $G \in R^q(\phi).$ By definition, we are able to find $(H^1_h) _h, (H^2_h) _h  \subset L^q(Y, \mm)$ such that
      \begin{enumerate}
\item $0\leq H^1_h(y) \leq G_h(y)$ for a.e. $y\in Y.$
\item There is a sequence $(\phi_{h,n})_{h,n} \subset \LL $ such that
\begin{equation*}
\phi _{h,n} \to_{n} \phi  \,\, \mbox{ in } L^q(Y,\R^s, \mm) \quad \mbox{ and } \quad Ils_a (\phi_{h,n}) (\cdot) \rightharpoonup _n H^2 _h\,\, \mbox{ in } L^q(Y, \mm).
\end{equation*}
\item $1\leq H^1_h(y) \leq  H^2_h(y),$  for a.e. $y\in Y.$ 
\end{enumerate}
Up to a subsequence, using the reflexivity of $L^q(Y, \mm)$ (see,  \cite[Theorem 4.10]{B10}), we can suppose $H^i_h \rightharpoonup H^i\in L^q(Y, \mm)$ with $i=1,2;$ then using a diagonal argument we find:
\begin{equation*}
\phi _{h,n(h)} \to_{h} \phi \,\, \mbox{ in } L^q(Y,\R^s, \mm) \quad \mbox{ and } \quad Ils_a (\phi_{h,n(h)}) (\cdot) \rightharpoonup _h H^2 \,\, \mbox{ in } L^q(Y, \mm).
\end{equation*}
On the other hand, $H^1(y) \leq G(y)$ and $H^1(y) \leq  H^2(y)$ for a.e. $y\in Y.$  This means that $G \in R^q(\phi),$ as desired.

Regarding the convexity of $R^q(\phi),$ let $G_1, G_2  \in R^q(\phi).$ We would like to show 
\begin{equation*}
tG_1 + (1-t) G_2 \in R^q(\phi) \quad \forall t\in [0,1].
\end{equation*}
Fix $t\in [0,1].$ By definition, we have $H_1 , H_2,  (\phi _h)_h$ for $G_1$ and $ F_1 , F_2,  (\psi _h)_h$ for $G_2$ as in Definition \ref{defIntrinsically relaxed slope}.

Hence,  it easy to see that $$\tilde H_1:= tH_1 + (1-t) F_1$$ and so the first point holds. Regarding the second point in Definition \ref{defIntrinsically relaxed slope}, we set 
\begin{equation*}
\tilde \phi _h  = t \phi _h +(1-t)  \psi _h \qquad \mbox{and} \qquad   \tilde H_2 =\frac c 2 ( H_2+F_2).
\end{equation*}
It is trivial that  $\tilde \phi _h \to \phi $ and, by Proposition \ref{sommadiLipeLip.2} and \ref{propconvex30} (2), $ Ils _a(\tilde \phi_{h})  \rightharpoonup \tilde H^2$ in $L^q(Y , \mm).$ Moreover, by Lemma \ref{lem22b}, $(\tilde \phi _h)_h \subset \LL.$


 Finally, by definition of $\tilde H_1, \tilde H_2$  and recall that $c\geq 2,$ we get the last point and, as a consequence, we deduce the convexity of $R^q(\phi).$

$\quad $

Now we show the second point of this statement. Let $G_1, G_2  \in R^q(\phi).$ We want to show that $ \min \{ G_1(y) , G_2 (y)\}  \in R^q(\phi).$ We consider $H_1 , H_2,  (\phi _h)_h$ for $G_1$ and $ F_1 , F_2,  (\psi _h)_h$ for $G_2$ as in Definition \ref{defIntrinsically relaxed slope}. Here, we choose 
\begin{equation*}\label{int}
\tilde H_1 (y)=\left\{ 
\begin{array}{lcl}
H_1(y)  &   & \mbox{ for $y\in Y$ such that } \min \{ G_1(y) , G_2 (y)\} = G_1(y), \\
F_1 (y) &    & \mbox{ for $y\in Y$ such that } \min \{ G_1(y) , G_2(y) \} = G_2(y). \\
\end{array}
\right.
\end{equation*}
and so the first point in Definition \ref{defIntrinsically relaxed slope} holds. Moreover,  set 
\begin{equation*}
\tilde \phi _h  = \frac 1 2 \phi _h +\frac 1 2  \psi _h \qquad \mbox{and} \qquad   \tilde H_2 =\frac c 2 ( H_2+F_2).
\end{equation*}
By Lemma \ref{lem22b}, $(\tilde \phi _h)_h \subset \LL $ and using Proposition \ref{sommadiLipeLip.2} and \ref{propconvex30} (2)
 we get the last two points in Definition \ref{defIntrinsically relaxed slope}.

Let's now prove the last part of the statement: since $R^q(\phi)$ is a closed and convex subset of $L^q(Y,\mm)$, it turns out to be weakly closed. Therefore weak lower semicontinuity and coercivity of the $L^q(Y,\mm)$ norm imply the existence of a minimal element in the weakly closed subset $R^q(\phi) \subset L^q(Y, \mm)$. Note that the pointwise minimality condition follows directly by absurd: suppose it is not true, that is we can find a $G \in R^q(\phi)$ such that:
\begin{equation*}
G< |D\phi|_q \mbox{ on a set of $\mm$-positive measure.}
\end{equation*}
But then we can define: $\tilde G:=  \min \{ G ,  |D\phi|_q \} $ and by the first part of the statement, it follows $\tilde G \in R^q(\phi)$. This gives a contradiction to the minimality of $ |D\phi|_q.$
  Hence the proof of the statement is complete.

    \end{proof}
  
 Notice that the choice of a sequence $(\tilde \phi _n)_n$ and $\tilde H_2$ in order to prove the convexity and the second point of last statement is quite independent to $G$ but it strongly depends on the Leibniz formula. Unfortunately, at least using this approach it does not seem possible to define a relationship between $ H_2 $ and $ G $ in order that, once the minimum relaxed slope has been defined, this is equal to its associated map $H_2$.

However, an easy consequence of Proposition \ref{Locality and Minimality} is the following relation between the minimal intrinsic relaxed slope of an intrinsically Lipschitz section and its Lipschitz constant.
  \begin{prop} Let $\phi \in \LL \cap L^q(Y, \mm),$ then
\begin{equation*}
|D\phi|_q \leq Ils_a (\phi)(y) \,\, \mm -a.e. \mbox{ in } Y.
\end{equation*}
\end{prop}

  \begin{proof}  The statement follows immediately from the fact that for an intrinsically Lipschitz and bounded section, its slope  is a relaxed slope for $\phi$. Indeed one can take the constant approximation $(\phi_{h})_h := \phi $ for any $h\in \N,$ in the definition of relaxed slope.
\end{proof}
  
   We conclude this section given a key notion in order to define intrinsic Cheeger energy.
  
    \begin{defi}[Minimal intrinsic $q$-relaxed slope of a section]\label{Minimal intrinsically relaxed slope of a section}  Let $q\in (1,+\infty)$ and $\phi \in L^q(Y,\R^s, \mm)$  be a section of   $\pi .$  If $R^q(\phi)$ is no empty, which means that $\phi$ has at least one relaxed slope, then 
    \begin{equation*}
|D\phi |_q \mbox{ is called minimal  intrinsic $q$-relaxed slope of } \phi.
\end{equation*}
\end{defi}

%

The following statement give a condition in order to have a strong convergence of the  minimal intrinsic $q$-relaxed slope.
     
        \begin{prop}\label{Strong convergence} Let $q\in (1,+\infty)$ and $\phi \in L^q(Y, \R^s, \mm)$ be a section of $\pi.$ Then, the following properties hold:
     \begin{enumerate}
     \item   If $G \in R^q(\phi)$ with $H_2$ as in Definition \ref{defIntrinsically relaxed slope},  then there is $(\phi_n)_n \subset \LL$ converging to $\phi$ in $L^q(Y, \R^s, \mm)$ and $(G_n)_n \subset L^q(Y, \mm)$ strongly convergent to $H_2$ in $L^q(Y, \mm).$ 
     \item If $(G_n)_n \subset R^q(\phi _n), \, (\phi_n)_n \subset \LL$ such that $\phi _ n \rightharpoonup \phi $ in  $L^q(Y,\R^s,  \mm)$ and $G_n \rightharpoonup G$ in  $L^q(Y, \mm)$, then $G \in R^q(\phi).$ 
        \item Let $K \subset Y$ be a compact subset and $\psi := \phi _{|K}.$ Then if $G \in R^q(\phi)$ it holds
        \begin{description}
\item[ 3.a] $G _{|K} \in R^q(\psi);$
\item[ 3.b]  there are $H_1, H_2 \in L^q(K, \mm)$ such that
  \begin{itemize}
\item $0\leq  H_1(y) \leq  G(y)$ and  $0\leq H_1(y) \leq H_2 (y)$ for a.e. $y\in K.$ 
\item  There is a sequence $ (\psi _h)_h \subset \LL $ such that
\begin{equation*}
\psi _h \to \psi   \,\, \mbox{ in } L^q(K,\R^s, \mm)\quad \mbox{ and } \quad  Ils_a ( \psi_h) (.) \to  H_2 \,\, \mbox{ in } L^q(K, \mm).
\end{equation*}
\end{itemize}
  
\end{description}

\end{enumerate}
     \end{prop}
   
     \begin{proof} The third point is trivial. For the other points, we follows Lemma 4.3 in \cite{AGS08.2}. 
     
     (1). By definition, we know that there are $H_1, H_2 \in L^q(Y, \mm)$ such that

  \begin{itemize}
\item $0\leq  H_1(y) \leq  G(y)$ and  $0\leq H_1(y) \leq H_2(y)$ for a.e. $y\in Y.$ 
\item  There is a sequence $ (\phi _h)_h \subset \LL $ such that
\begin{equation*}
\phi _h \to \phi  \,\, \mbox{ in } L^q(Y,\R^s, \mm) \quad \mbox{ and } \quad  Ils _a( \phi_h) (.) \rightharpoonup  H_2 \,\, \mbox{ in } L^q(Y, \mm),
\end{equation*}
\end{itemize}
By Mazur's lemma (see Corollary 3.8 in \cite{B10}) we can find a sequence of convex combination $Ils_h$ of $ Ils_a ( \phi_h)$, starting from an index  $h(n) \to \infty$ strongly convergent to $H_2$ in $L^q(Y, \mm).$ The corresponding convex combinations of $\phi_h,$ that we shall denote by $\psi _h,$ still converge  to $\phi$ in $L^q(Y, \mm)$ and belongs to $\LL$ (see Lemma \ref{lem22b}).

(2). We need to prove that the set 
\begin{equation*}
R:=\{ (\phi , G) \in L^q(Y,\R^s,  \mm)\times L^q(Y, \mm) \, :\, \phi \in \LL , \, G \mbox { is a relaxed slope of } \phi\}.
\end{equation*} is weak closed in $L^q(Y, \R^s, \mm)\times L^q(Y, \mm).$ If we show that (2.a) $R$ is closed, it is sufficient to prove that (2.b) $R$ is strongly closed. 

(2.a). Let $(\phi _1, G_1), (\phi _2, G_2) \in R.$ We want that $(t\phi _1 + (1-t)\phi _2, tG_1 + (1-t)G_2) \in R$ for every $t\in [0,1].$ This follows immediately from   Proposition \ref{Locality and Minimality}, Lemma \ref{lem22b} and from the fact $L^q$ is convex.

(2.b). If  $(\phi _n , G_n)_n \subset R$ strongly converge to $(\phi , G)$ in $L^q(Y,\R^s, \mm)\times L^q(Y, \mm),$ we want to show that $(\phi , G) \in R.$  We can find sequences $(\psi _{n,h})_{n,h} \subset  \LL \cap L^q(Y, \mm)$ and $G_{n,h} \subset L^q(Y, \mm)$ such that
\begin{itemize}
\item $\psi _{n,h} \to \phi _{n}$ in $L^q(Y, \R^s, \mm)$ with $ Ils_a (\psi _{n,h}) \rightharpoonup_h  H_{2,n}$ in $L^q(Y, \mm);$ 
\item  $G_{n,h} \to  H_{2,n}$ in  $L^q(Y, \mm)$ with $0\leq H_{1,n} \leq G_n$ and $0\leq H_{1,n} \leq H_{2,n}$ a.e. $y\in Y.$
\end{itemize}
Possibly extracting a suitable subsequence, we can assume that $ H_{2,n} \rightharpoonup H_2 $ in $L^q(Y, \mm).$  Moreover, using the reflexivity of $L^q(Y, \mm),$ we get that, up to subsequences, $H_{1,n} \rightharpoonup H_1$ in $L^q(Y, \mm)$ (see,  \cite[Theorem 3.18]{B10}).
Hence, by standard diagonal argument we can find an increasing sequence $h \mapsto n(h)$ such that $\psi _{n,h(n)} \to \phi $  in $L^q(Y,\R^s, \mm)$ and $ Ils (\psi _{n,h(n)}) \rightharpoonup H_2$ in $L^q(Y, \mm)$  with $0\leq H_{1} \leq G$ and $0\leq H_{1} \leq H_{2}$ a.e. $y\in Y.$ Hence $(\phi , G ) \in R,$ as desired.

\end{proof}

   \begin{rem}
   We do not know if it is possible to get Proposition \ref{Strong convergence} (3) without the condition of compactness of $K.$ On the other hand it is no possible to apply the proof of Theorem 2.2.7 in \cite{P17} because the inequality (2.2.12) not works in our intrinsic context. The motivation is given by our $non \, usual$ Leibniz formula (see Proposition  \ref{sommadiLipeLip.2}).
   \end{rem}

     \subsection{Intrinsic Cheeger energy: definition} In this section  we consider $(Y,d,\mm)$ a  metric measure space and  $\pi :\R^s \to Y$   a  quotient and linear map. Set 
   \begin{equation}
\LL  :=  S \cap ILS_b(Y),
\end{equation}
where $S$ is the set of all   intrinsically $L$-Lipschitz sections of $\pi$  satisfying the inequality \eqref{equationSobolev10} for some $c\geq 2$.
 Using Proposition \ref{Locality and Minimality} for $q=2,$ we define the intrinsic Cheeger energy as follows. 

    \begin{defi}[Intrinsic Cheeger energy]\label{defCheeger energy}  The intrinsic Cheeger energy  is the functional $$iCh _a : L^2(Y , \mm) \to [0, +\infty]$$ defined as 
\begin{equation*}
\begin{aligned}
iCh_a (\phi) & := \inf\left\{ \liminf _{ n \to +\infty} \int_Y Ils _a^2 (\phi _n )(y)\, d\mm (y) \, :\, (\phi_n)_n \subset \LL,  \phi _n \to \phi \mbox{ in } L^2(Y, \R^s, \mm) \right\}.
\end{aligned}
\end{equation*}
The domain of the intrinsic Cheeger energy  is the subset of $ L^2(Y, \mm)$ defined by:
\begin{equation*}
\begin{aligned}
Dom(iCh_a) & :=\left\{ \phi \in L^2(Y, \R^s, \mm) \,:\,  iCh_a (\phi)  <+ \infty \right\}.\\
\end{aligned}
\end{equation*}
One other possibility could be to use the slope instead of the asymptotic Lipschitz constant, namely:
\begin{equation*}
\begin{aligned}
iCh_{Ils} (\phi) & := \inf\left\{ \liminf _{ n \to +\infty} \int_Y Ils ^2 (\phi _n)( y)\, d\mm (y) \, :\, (\phi_n)_n \subset \LL,  \phi _n \to \phi \mbox{ in } L^2(Y, \R^s, \mm) \right\}.
\end{aligned}
\end{equation*}
\end{defi}
  
 Remark \ref{remperchegeerdopo} directly implies the easier inequality:
 \begin{equation*}
iCh_a (\phi) \geq iCh_{Ils} (\phi),\quad \forall \phi \in Dom(iCh_a).
\end{equation*}

This two definitions of intrinsic Cheeger energy are build through a relaxation process; indeed, $iCh_a$ is the relaxation of the functional
\begin{equation*}
\Gamma _a(\phi):= \left\{
\begin{array}{l}
\liminf _{ n \to +\infty} \int_Y Ils_a ^2 (\phi _n)( y)\, d\mm (y), \quad \qquad \,\,  \mbox{if } \phi \in \LL \\
+\infty ,  \qquad \qquad  \qquad  \qquad \qquad  \qquad  \qquad \qquad  \mbox{if } \phi \in L^2(Y, \R^s, \mm) - \LL
\end{array}
\right.
\end{equation*}
with respect to the $L^2(Y, \mm)$ distance, in the sense that it is the biggest lower semicontinuous functional which is less or equal than $\Gamma _a.$ In a similar way, $iCh_{Ils}$ is the relaxation of the functional 
\begin{equation*}
\Gamma _{Ils}(\phi):= \left\{
\begin{array}{l}
\liminf _{ n \to +\infty} \int_Y Ils ^2 (\phi _n)( y)\, d\mm (y), \quad \qquad \,\,  \mbox{if } \phi \in \LL \\
+\infty ,   \qquad \qquad  \qquad  \qquad \qquad  \qquad  \qquad \qquad \mbox{if } \phi \in L^2(Y,\R^s, \mm) - \LL
\end{array}
\right.
\end{equation*}

In the classical case, it is well known that the space of Lipschitz maps is dense in $L^2(Y)$ but in the intrinsic case we do not have this property. However, we recall that  we get the trivial case when $s=1$.

      \subsection{Intrinsic Cheeger energy: equivalent condition} 
   The set $R^2(\phi)$ given by Definition \ref{defIntriRquadro} represents the relaxation on the asymptotic Lipschitz constant instead of as above where we do a relaxation on the functional $\Gamma _a$ and it is the key notion in order to give the following result.    
  \begin{theorem}\label{theorem30may}
  Let $\phi \in L^2(K,\R^s, \mm)$ with $K \subset Y$ compact and suppose that $R^2(\phi)$ is no empty. Then, we have the following representation formula of the Cheeger energy:
  \begin{equation*}
iCh_a(\phi)= \int _KH_2^2(y) \, d\mm (y),
\end{equation*}
where $H_2$ is given by Definition \ref{defIntrinsically relaxed slope} for $|D\phi|_2.$
  \end{theorem}
  
  \begin{proof}
We follows  the last part of the proof in \cite[Theorem 2.2.7]{P17}. Without loss of generality, we can suppose $\phi \in Dom (iCh_a).$ Hence, by definition of intrinsic Cheeger energy and Proposition \ref{Strong convergence} (3) applying to $|D\phi|_2$, we have that
\begin{equation*}
iCh_a(\phi) \leq  \liminf _{ n \to +\infty} \int_K Ils_a ^2 (\phi _n )(y)\, d\mm (y) = \int _K H_2^2(y) \, d \mm (y),
\end{equation*}
and this provides the upper bound on $iCh_a$. On the other hand, let $(\psi_n)_n \subset \LL$ be a sequence  such that 
\begin{itemize}
\item $\psi _n \to \phi$ in $L^2(K,\R^s, \mm);$
\item $ \liminf _{ n \to +\infty} \int_K Ils_a^2  (\psi _n )( y)\, d\mm (y) <\infty .$
\end{itemize}
 By reflexivity of $L^2$, we get that, up to a subsequence, $(Ils_a^2  (\psi _n))_n$ converges weakly in $L^2(K, \mm)$ and so also strongly since $K$ is compact. Finally, using definition of minimal intrinsic relaxed slope, we obtain
 \begin{equation*}
\liminf _{ n \to +\infty} \int_K Ils_a^2  (\psi _n )(y)\, d\mm (y)  \geq \int _ K H_2^2(y) \, d\mm (y).
\end{equation*}
Since $(\psi_n)_n$ is generic we end up with
\begin{equation*}
iCh_a(\phi) \geq  \int _K H_2^2(y) \, d \mm (y),
\end{equation*}
and we are done.
  \end{proof}

 \bibliographystyle{alpha}
\bibliography{DDLD}

\end{document}